\documentclass[a4paper,11pt]{article}
\usepackage[T1]{fontenc}
\usepackage[utf8]{inputenc}
\usepackage[english]{babel}
\usepackage[margin=1.1in]{geometry}
\usepackage{microtype}
\usepackage{braket}
\usepackage{amsmath}
\usepackage{amssymb}
\usepackage{amsthm, aliascnt}
\usepackage{mathtools}
\usepackage{mathrsfs}
\usepackage{xcolor}
\usepackage[hyphens]{url}
\usepackage{hyperref}
\usepackage[hyphenbreaks]{breakurl}
\usepackage{enumitem}
\usepackage{tikz}
\usepackage{csquotes}
\usetikzlibrary{patterns}
\usepackage{esint}
\usepackage{comment}

\hypersetup{
                colorlinks=true,
                linkcolor={magenta},
                citecolor={cyan},
                breaklinks=true}

\theoremstyle{plain}
\newtheorem{teor}{Theorem}
\numberwithin{teor}{section}
\numberwithin{equation}{section}
\theoremstyle{definition}
\newaliascnt{defi}{teor}
\newtheorem{defi}[defi]{Definition}
\aliascntresetthe{defi}

\theoremstyle{plain}
\newaliascnt{lemma}{teor}
\newtheorem{lemma}[lemma]{Lemma}
\aliascntresetthe{lemma}
\theoremstyle{plain}
\newaliascnt{prop}{teor}
\newtheorem{prop}[prop]{Proposition}
\aliascntresetthe{prop}
\theoremstyle{plain}
\newaliascnt{cor}{teor}

\aliascntresetthe{cor}
\theoremstyle{plain}
\newaliascnt{conj}{teor}
\newtheorem{conj}[conj]{Conjecture}
\aliascntresetthe{conj}
\theoremstyle{definition}
\newaliascnt{ex}{teor}

\aliascntresetthe{ex}
\theoremstyle{definition}
\newaliascnt{oss}{teor}
\newtheorem{oss}[oss]{Remark}
\aliascntresetthe{oss}
\theoremstyle{plain}

\DeclarePairedDelimiter{\abs}{\lvert}{\rvert}
\DeclarePairedDelimiter{\norma}{\lVert}{\rVert}

\DeclareMathOperator{\diam}{diam}
\newcommand{\R}{\mathbb{R}}

\newcommand{\Hn}{\mathcal{H}^{n-1}}

\newcommand{\eps}{\varepsilon}

\makeatletter
\newcommand{\leqnomode}{\tagsleft@true\let\veqno\@@leqno}
\newcommand{\reqnomode}{\tagsleft@false\let\veqno\@@eqno}

\newcommand{\ddfrac}[2]{\frac{\displaystyle #1}{\displaystyle #2}}

\newcommand{\prof}{\mathcal{P}}

\usepackage[maxbibnames=99]{biblatex}

\AtEveryBibitem{%
    \clearfield{issn} 
    \clearfield{isbn} 
    \clearfield{doi}
    \ifentrytype{online}{}{
    \clearfield{url}
  }
}
\makeatother

\title{Estimates on the Neumann and Steklov principal eigenvalues of collapsing domains}
\author{P. Acampora, V. Amato, E. Cristoforoni}
\date{}
\newcommand{\Addresses}{{
 \bigskip 
 \footnotesize 
 
 \textsc{Dipartimento di Matematica e Applicazioni ``R. Caccioppoli'', Universit\`a degli studi di Napoli Federico II, Via Cintia, Complesso Universitario Monte S. Angelo, 80126 Napoli, Italy.}\par\nopagebreak 
 
 \medskip 
 
 \textit{E-mail address}, P.~Acampora: \texttt{paolo.acampora@unina.it} 
  
 \medskip

 \textsc{
 	Holder of a research grant from Istituto Nazionale di Alta Matematica "Francesco Severi" at Dipartimento di Matematica e Applicazioni "R. Caccioppoli", Via Cintia, Complesso Universitario Monte S. Angelo, 80126 Napoli, Italy.}

 \medskip 
   \footnotesize 
  \textit{E-mail address}, V. ~Amato: \texttt{amato@altamatematica.it} 

  \medskip

\textsc{Mathematical and Physical Sciences for Advanced Materials and Technologies, Scuola Superiore Meridionale, Largo San Marcellino 10, 80126, Napoli, Italy.}\par\nopagebreak 
 
 \medskip 
 
 \textit{E-mail address}, E.~Cristoforoni: \texttt{emanuele.cristoforoni@unina.it} 
}}

\begin{document}
\reversemarginpar
\maketitle
\begin{abstract}
 We investigate the relationship between the Neumann and Steklov principal eigenvalues emerging from the study of collapsing convex domains in $\mathbb{R}^2$. Such a relationship allows us to give a partial proof of a conjecture concerning estimates of the ratio of the former to the latter: we show that thinning triangles maximize the ratio among convex thinning sets, while thinning rectangles minimize the ratio among convex thinning with some symmetry property. 
 \end{abstract}
\renewcommand*{\sectionautorefname}{Section}
\renewcommand*{\subsectionautorefname}{Subsection}
\section{Introduction}
Let $\Omega\subset \R^2$ be a bounded, open, connected and Lipschitz set. We define the Neumann and Steklov eigenvalues as follows: find positive constants $\mu, \sigma$ such that there exist non-zero solutions to the boundary value problems
\[
\begin{dcases}
-\Delta u=\mu u &\text{in }\Omega, \\
\frac{\partial u}{\partial \nu}=0 &\text{on }\partial\Omega,
\end{dcases}
\qquad \qquad \qquad
\begin{dcases}
\Delta v=0 &\text{in }\Omega, \\
\frac{\partial v}{\partial \nu}=\sigma v &\text{on }\partial\Omega.
\end{dcases}
\]
The regularity assumption we made on $\Omega$ ensures that we can find two increasing and divergent sequences of eigenvalues 
\[
\begin{split}
0=\mu_0(\Omega) < \mu_1(\Omega)\le \mu_2(\Omega)\le \dots \le \mu_k(\Omega)\le \dots, \\[5 pt]
0=\sigma_0(\Omega) < \sigma_1(\Omega)\le \sigma_2(\Omega)\le \dots \le \sigma_k(\Omega)\le \dots,
\end{split}
\]
which are the spectrum of the Neumann laplacian and the spectrum of the Dirichlet-to-Neumann map respectively. We recall the variational characterization of the eigenvalues, for $k\geq 0$:
\[
\mu_k(\Omega)=\inf_{E\in  \mathcal{S}_{k+1}(\Omega)}\,\,\sup_{w\in E\setminus \{0\}}\ddfrac{\int_\Omega \abs{\nabla w}^2\,dx}{\int_\Omega w^2\,dx}, \qquad \qquad \sigma_k(\Omega)=\inf_{\substack{E\in  \mathcal{S}_{k+1}(\Omega)}}\,\,\sup_{ w\in  E\setminus \{0\}}\ddfrac{\int_\Omega \abs{\nabla w}^2\,dx}{\int_{\partial\Omega} w^2\,d\Hn},
\]
where $\mathcal{S}_{k+1}(\Omega)$ is the family of all linear subspaces of $H^1(\Omega)$ of dimension $k+1$. In particular, we are interested in the principal eigenvalues, i.e. $k=1$, namely

\noindent\begin{minipage}{0.5\linewidth}
\begin{equation*}
  \mu_1(\Omega)=\inf_{\substack{w\in H^1(\Omega)\setminus\{0\} \\\int_\Omega w=0 }}\ddfrac{\int_\Omega \abs{\nabla w}^2\,dx}{\int_\Omega w^2\,dx}, 
\end{equation*}
\end{minipage}%
\begin{minipage}{0.5\linewidth}
\begin{equation*}
\sigma_1(\Omega)=\inf_{\substack{w\in H^1(\Omega)\setminus\{0\} \\ \int_{\partial \Omega}w= 0}}\ddfrac{\int_\Omega \abs{\nabla w}^2\,dx}{\int_{\partial\Omega} w^2\,d\Hn}.
\end{equation*}
\end{minipage}\par\vspace{\belowdisplayskip}
Many authors in the literature identified remarkable similarities between the two families of eigenvalues. Moreover,  an underlying relationship holds between the two quantities. For instance, Steklov eigenvalues can be seen as limits of weighted Neumann eigenvalues, while Neumann eigenvalues can be obtained as limits of Steklov eigenvalues by suitably perforating the set $\Omega$. We refer, for instance, to \cite{LP15}, and \cite{GHL21} for these results. We want to explore the same relationship between the two eigenvalues, from the shape optimization point of view.

Namely, we could be interested in the scale invariant ratio
\[
F(\Omega)=\frac{\abs{\Omega}\mu_1(\Omega)}{P(\Omega)\sigma_1(\Omega)},
\]
and, consequently, in the two problems
\begin{equation}
\label{problema0}
\min_{\Omega\in \mathcal{K}}F(\Omega),\qquad\qquad \max_{\Omega\in \mathcal{K}}F(\Omega),
\end{equation}
where $\mathcal{K}$ is a suitable class of sets, $\abs{\cdot}$ denotes the area, and $P(\cdot)$ denotes the perimeter. Unfortunately, the choice
\[
\mathcal{K}=\Set{\Omega\subset\R^2 | \Omega \text{ bounded, open and Lipschitz}}
\]
causes the problems in \eqref{problema0} to be ill-posed, in the sense that
\[
\inf_{\mathcal K}F(\Omega)=0, \qquad\qquad \sup_{\mathcal{K}}F(\Omega)=+\infty,
\]
as shown in \cite{HM22}, \cite{BHM21}, and \cite{GKL21}.

In order to obtain some comparison between Neumann and Steklov eigenvalues, we address the problems in \eqref{problema0} restricting the class of admissible sets to
\begin{equation}
\label{eq: Kconvex}
\mathcal{K}_c=\Set{\Omega\subset\R^2 | \Omega\text{ bounded, open and convex}}\!.
\end{equation}
This choice of $\mathcal{K}_c$ avoids shapes that could make $F$ degenerate, and precisely it could be shown, as in \cite{HM22}, that there exist two constants $c,C>0$ such that
\[
c\le F(\Omega)\le C \qquad \forall \, \Omega\in\mathcal{K}_c.
\]
Additionally, numerical simulations lead the authors to state the following
\begin{conj}[Henrot, Michetti \cite{HM22}]
\label{conje}
    Let $\mathcal{K}_c$ be as in \eqref{eq: Kconvex}, then
    \[
        1< F(\Omega)< 2 \qquad \forall \, \Omega\in\mathcal{K}_c.
    \]
    Moreover, the inequalities are sharp in the following sense: there exists a sequence $R_n$ of thinning rectangles, and a sequence $T_n$ of thinning triangles such that
    \[
        \lim_n F(R_n)=1, \qquad \qquad \lim_n F(T_n)=2.
    \]
\end{conj}

The aim of this paper is to take steps towards proving the conjecture; however, we do not provide an exhaustive solution.

The numerical simulations which support \autoref{conje} also suggest that the infimum and the supremum of $F(\Omega)$, in the class $\mathcal{K}_c$, are asymptotically achieved by particular sequences of thinning domains. Therefore we focus on the limits of $F(\Omega_\eps)$, where $\Omega_\eps$ is a family of thinning domains of the type \eqref{omegaeps}.
Indeed, following in the footsteps of \cite{HM22}, for such a family, 
there exists a non-negative concave function $h:[0,1]\to \R$ such that
\[
\lim_{\eps\to0}\mu_1(\Omega_\eps)=\mu_1(h) \qquad \qquad \lim_{\eps\to0} \frac{P(\Omega_\eps)\sigma_1(\Omega_\eps)}{\abs{\Omega_\eps}}=\sigma_1(h)\left(\int_0^1 h(t)\,dt\right)^{-1},
\]
where $\mu_1(h)$ is the first eigenvalue of the Sturm-Liouville problem
\begin{equation}\label{eq:Neumann}
    \begin{cases} -\dfrac{d}{dx}\left(h(x)\dfrac{d v}{dx}(x)\right)=\mu_1(h)h(x)v(x)\qquad &x\in(0,1),\\[15 pt]
    h(0)\dfrac{d v}{dx}(0)=h(1)\dfrac{ d v}{dx}(1)=0,
    \end{cases}
\end{equation}
while $\sigma_1(h)$ is the first eigenvalue of the Sturm-Liouville problem
\begin{equation}\label{eq:Steklov}
    \begin{cases} -\dfrac{d}{dx}\left(h(x)\dfrac{d v}{dx}(x)\right)=\sigma_1(h)v(x)\qquad &x\in(0,1),\\[15 pt]
    h(0)\dfrac{d v}{dx}(0)=h(1)\dfrac{ d v}{dx}(1)=0.
    \end{cases}
\end{equation}
The function $h$, in some sense, represents the profile of the thinning sets $\Omega_\eps$, and, in particular, we have that $h\equiv 1$ represents the limit of a family of thinning rectangles. On the other hand, for every $x_0\in(0,1)$, we let
\[
T_{x_0}(x):=\begin{cases} \dfrac{2x}{x_0} &x\in[0,x_0),\\[15 pt]
\dfrac{2(1-x)}{1-x_0} &x\in[x_0,1],
\end{cases}
\]
and 
\[
T_0(x)=2(1-x), \qquad \qquad T_1(x)=2x,
\]
which represents the limit of a family of thinning triangles. Consequently, familiarizing oneself with the properties of $\mu_1(h)$ and $\sigma_1(h)$ can offer advantages when it comes to analyzing the eigenvalues $\mu_1(\Omega)$ and $\sigma_1(\Omega)$. It is worth mentioning that the quantities $\mu_1(h)$ and $\sigma_1(h)$ are in a way related to a weighted Hardy constant (see \cite{KPS17}, \cite{OK90}, \cite{PW60}, and \autoref{prop: sigma=hardy}). 

Following this path, we refer to \cite{PW60}, \cite{HM22} and \cite{T65} for the proof of the subsequent properties: let 
\[
\prof=\Set{h\in L^\infty(0,1)\colon\, h \text{ non negative, concave and not identically zero}},
\] 
and
\[
\prof_1=\Set{h\in\prof | \int_0^1h(t)\,dt=1},
\]
then for every $h\in \prof_1$, we have that
\[
\begin{split}
\pi^2=\mu_1(1)\le \,\, &\mu_1(h)\le\mu_1(T_{1/2})
\\[5 pt]
&\sigma_1(h)\le \sigma_1(p)=12,
\end{split}
\]
where $p$ is the arc of parabola $p(x)=6x(1-x)$.\medskip

Here we state the main results of this work

\vspace{-3mm}
\begin{teor}
\label{teor: minstek}
The minimum problem 
\begin{equation}
    \label{eq:minstek}
    \min_{h\in \prof_1}\sigma_1(h)
\end{equation}
admits the functions $T_0$ and $T_1$ as unique solutions.
\end{teor}
We prove the theorem following two distinct approaches. \autoref{sec: stek1} is devoted to the former, while \autoref{sec: stek2} is devoted to the latter, which relies on a rearrangement method that, up to our knowledge, appears to be new. Finally, in \autoref{sec: mi/sigma} we establish a relationship between $\mu_1(h)$ and $\sigma_1(h)$.
\begin{teor}
\label{prop: mu=sigma}
    There exists an invertible operator
    \[
    \mathcal{G}:\prof\to \prof
    \]
    such that, for every $h,k\in\prof$, we have
    \begin{equation}
    \label{eq: neumtosigma}
    \left(\int_0^1 h(t)\,dt\right)^2 \mu_1(h)=\sigma_1\left(\mathcal{G}(h)\right),
    \end{equation}
    and
    \begin{equation}
    \label{eq: sigmatoneum}
    \left(\int_0^1 \frac{1}{\sqrt{k(t)}}\,dt\right)^{2}\sigma_1(k)=\mu_1(\mathcal{G}^{-1}(k)).
    \end{equation}
\end{teor}

It may help to solve problems obtained by studying \eqref{problema0} among thinning domains, namely
\begin{equation*}
\min_{h\in \prof}\ddfrac{\mu_1(h)\int_0^1 h(t)\,dt}{\sigma_1(h)}, \qquad \qquad \max_{h\in\prof}\ddfrac{\mu_1(h)\int_0^1 h(t)\,dt}{\sigma_1(h)}.
\end{equation*}
In particular, we can fully solve the maximizing problem, and partially solve the minimizing problem. We summarize these results in the following theorem.

\begin{teor}\label{teor: mu/sigma}
    Let $h\in \prof_1$. Then 
    \[
        \frac{\mu_1(h)}{\sigma_1(h)}\le 2,
    \]
   and the equality holds if and only if $h=T_{x_0}$ for some $x_0\in [0,1]$. If, in addition, $h(x)=h(1-x)$ for every $x\in[0,1]$, then
   \[
        \frac{\mu_1(h)}{\sigma_1(h)}\ge 1.
   \]
\end{teor}
\section{Notations and tools}

Here we define standard quantities for convex sets and the formal definition of thin domain. This definition passes through the ones of  support function and  minimal width (or thickness).

We refer to \cite{HM22} for the proof of the lemmas in this section.

\begin{defi}
\label{support}
  Let $\Omega \subset \R^N$ be a bounded, open, and convex set. We define the \emph{support function} of $\Omega$ as
  \begin{equation*}
    h_\Omega(y)=\sup_{x\in \Omega}\left(x\cdot y\right), \qquad y\in \mathbb{R}^n .
  \end{equation*}
\end{defi}

\begin{defi}
  Let $\Omega \subset \R^N$ be a bounded, open and convex set, and  let $y \in \mathbb{R}^n$. We define the \emph{width} of $\Omega$ in the direction $y$ as 
  \begin{equation*}
    \omega_{\Omega}(y)=h_{\Omega}(y)+h_{\Omega}(-y)
    \end{equation*}
 and we define the \emph{minimal width} of $\Omega$ as
\begin{equation*}
    w_\Omega=\min\{  \omega_{\Omega}(y)\,|\; y\in\mathbb{S}^{n-1}\}.
\end{equation*}
\end{defi}
Hence, if $\diam(\Omega)$ denotes the diameter of $\Omega$, then we have
\begin{defi}
     Let $\Omega_\delta\subset \R^n$ be a family of non-empty, bounded, open, and convex sets. We say that $\Omega_\delta$ is a family of \emph{thinning domains} if
    \begin{equation*}
        \lim_{\delta\to 0}\, \dfrac{w_{\Omega_\delta}}{\diam(\Omega_\delta)}=0.
    \end{equation*}
     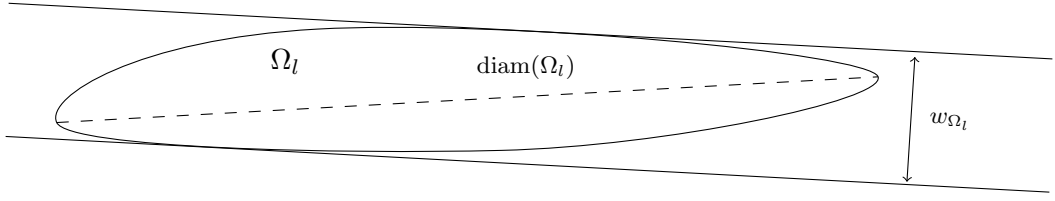
\begin{figure}[h!]
\begin{center}
 \begin{tikzpicture}[x=0.75pt,y=0.75pt,yscale=-1,xscale=1]
\draw   (147.5,317) .. controls (141.5,304) and (187.5,275) .. (277.5,270) .. controls (367.5,265) and (553.5,282) .. (557.5,294) .. controls (561.5,306) and (450.5,329) .. (369.5,331) .. controls (288.5,333) and (153.5,330) .. (147.5,317) -- cycle ;
\draw    (123.5,257) -- (644.5,285) ;
\draw    (122,324) -- (643,352) ;
\draw   [<->]  (575.87,284.05) -- (571.9,346.79) ;
\draw  [dash pattern={on 4.5pt off 4.5pt}]  (147.5,317) -- (557.5,294) ;

\draw (356,282) node [anchor=north west][inner sep=0.75pt]   [align=left] {\footnotesize $\diam(\Omega_l)$};
\draw (582,311) node [anchor=north west][inner sep=0.75pt]   [align=left] {\footnotesize $w_{\Omega_l}$};
\draw (253,279) node [anchor=north west][inner sep=0.75pt]   [align=left] {$\Omega_l$};
\end{tikzpicture}
\end{center}
\caption{Minimal width and diameter of a convex set.} \label{fig:M1}
\end{figure}
\end{defi}

Let us now consider a particular family of thinning domains. Let $h^+,h^-\in \prof$ such that $h^++h^-\in\prof_1$. We consider the family of thinning domains 
\begin{equation}
    \label{omegaeps}
    \Omega_\varepsilon= \Set{(x,y) \in \R^2 | \begin{aligned}\,  0 \leq \,&x\leq 1, \\ -\varepsilon h_-(x) \leq \,&y\leq \varepsilon h_+(x)\end{aligned}}.
\end{equation}
For such a sequence we have that both the principal eigenvalues of the Neumann and Steklov problems converge to the principal eigenvalues of the Sturm-Liouville problems \eqref{eq:Neumann} and \eqref{eq:Steklov} respectively. 
More precisely, if we define

\noindent\begin{minipage}{0.5\linewidth}
\begin{equation}
\label{eigenmu}
  \mu_1(h)=\inf_{\substack{u\in H^1(0,1) \\\int_0^1uh\,dx=0}}\dfrac{\displaystyle\int_0^1  (u')^2 h\,dx}{\displaystyle\int_0^1 u^2 h\,dx},
\end{equation}
\end{minipage}%
\begin{minipage}{0.5\linewidth}
\begin{equation}
\label{eigensigma}
 \sigma_1(h)=\inf_{\substack{v\in H^1(0,1) \\\int_0^1v\,dx=0}}\dfrac{\displaystyle\int_0^1  (v')^2h\,dx}{\displaystyle\int_0^1 v^2 \,dx},
\end{equation}
\end{minipage}\par\vspace{\belowdisplayskip}
we have the following lemmas
\begin{lemma}
\label{lem: limitproblem}
Let $\{\Omega_\varepsilon\}$ be family of thinning domains as in \eqref{omegaeps} and let $h=h_{-}+h_{+}$.
Then 
\begin{align*}
    \mu_1(\Omega_\varepsilon) &= \mu_1(h) + o(1) \text{ as } \varepsilon\to 0,\\
    \sigma_1(\Omega_\varepsilon) &= \frac{\sigma_1(h)}{2}\varepsilon + o(\varepsilon) \text{ as } \varepsilon\to 0.
\end{align*}
\end{lemma}

The following compactness result for $\mathcal{P}$ holds true
\begin{lemma}
\label{lem: compact}
    Let $h_n\in\prof_1$ be a sequence of functions, then there exists $h\in\prof$ such that, up to a subsequence, we have:
    \begin{itemize}
        \item $h_n$ converges to $h$ in $L^2(0,1)$;
        \item $h_n$ converges to $h$ uniformly on every compact subset of $(0,1)$.
    \end{itemize}
\end{lemma}

We also recall a continuity property of the eigenvalues $\mu_1(h)$ and $\sigma_1(h)$.

\begin{lemma}
\label{lem: continuity}
Let $h_n,h\in\prof$ be a sequence such that $h_n$ converges in $L^2(0,1)$ to $h$. Then we have
\begin{align*}
    \lim_{n}\mu_1(h_n)= \mu_1(h),\\
    \lim_{n}\sigma_1(h_n)= \sigma_1(h).
\end{align*}
\end{lemma}

\subsection{Other tools}
Here we recall some other tools that will be useful in the next pages. We refer to \cite{M12,AFP00}.


 \begin{teor}[Coarea formula]
 \label{coarea}
 Let $\Omega \subset \mathbb{R}^n$ be an open set with Lipschitz boundary. Let $f\in W^{1,1}_{\text{loc}}(\Omega)$, and let $u:\Omega\to\R$ be a measurable function. Then,
 \begin{equation*}
   {\displaystyle \int _{\Omega}u(x)|\nabla f(x)|dx=\int _{\mathbb {R} }dt\int_{\Omega\cap f^{-1}(t)}u(y)\, d\mathcal {H}^{n-1}(y)}.
 \end{equation*}
 \end{teor}

Here we define the notion of decreasing rearrangement
\begin{defi}
	Let $\Omega\subset\R^n$ be an open set, and let $u: \Omega \to \R$ be a measurable function. We define the \emph{distribution function} $\eta_u : [0,+\infty[\, \to [0, +\infty[$ of $u$ as the function
	$$
	\eta_u(t)= \abs{\Set{x \in \Omega \, :\,  \abs{u(x)} > t}}
	$$
 \end{defi}
\begin{defi}\label{defi:incre}
	Let $u: \Omega \to \R$ be a measurable function. We define the \emph{increasing rearrangement} $u_*$ of $u$ as 
 \[
 u_*(s)=\inf\Set{t>0 | \eta(t)\leq\abs{\Omega}-s}.
 \] 

\end{defi}
\begin{oss}
Let $\Omega\subset\R^n$ be an open set, and let $u: \Omega \to \R$ be a measurable function. Then $u$ and its increasing rearrangement $u_*$ are equi-measurable namely
\[
\eta_u=\eta_{u_*},
\]
and, in addition, for every $p\in[1,+\infty)$,
$$
\displaystyle{\norma{u}_{L^p(\Omega)}=\norma{u_*}_{L^p(0, \abs{\Omega})}}.
$$
\end{oss}

Finally, here is an important property of extreme points convex sets.

\begin{defi}
Let $V$ be a vector space, let $C\subset V$ be a convex set, and let $z\in C$. We say that $z$ is an \emph{extreme point} of $C$ if it cannot be written as a convex combination of distinct elements of $C$. More precisely, if $z=(1-t)x+ty$, with $x,y\in C$ and $t\in(0,1)$, then $x=y=z$.
\end{defi}

\begin{prop}
\label{prop: triextr}
    Let $h\in\prof_1$. Then $h$ is an extreme point for $\prof_1$ if and only if there exists $x_0\in[0,1]$ such that $h=T_{x_0}$.
\end{prop}
\begin{proof}
    Let us start by proving that for every $x_0\in[0,1]$ the triangle $T_{x_0}$ is an extreme point of $\prof_1$. Let $h\in\prof_1$ and let $x_M$ be a maximum point for $h$, then the concavity of $h$ ensures \[h\ge \dfrac{h(x_M)}{2}T_{x_M}.\]
    Recalling that $\int_0^1 h\,dx=1$, we get that 
    \begin{equation}
    \label{eq: h<2}
    h(x_M)=\max_{[0,1]}h\le 2,
    \end{equation}
    and the equality holds if and only if $h=T_{x_M}$. 
    
    Let now $x_0\in[0,1]$, and assume that
    \[
    T_{x_0}(x)=(1-t)h_0(x)+th_1(x) \qquad x\in[0,1],
    \]
    with $h_0,h_1\in\prof_1$, and $t\in(0,1)$. Since 
    \[
    2=T_{x_0}(x_0)\le\max\{h_0(x_0),h_1(x_0)\},
    \]
    and
    \[
    2=(1-t)h_0(x_0)+th_1(x_0),
    \]
    we get equality in \eqref{eq: h<2} for both $h_0$ and $h_1$. Therefore,  $h_0=h_1=T_{x_0}$, and we have proved that $T_{x_0}$ is an extreme point of $\prof_1$.\medskip

    We now prove that the triangles are the only extreme points of $\prof_1$. Let $h\in\prof_1$ be such that $h\ne T_{x_0}$ for every $x_0\in [0,1]$.
    
    We begin by assuming that $h(1)>0$. Notice that, in this setting, there exists $s\in(0,1)$, such that the function 
    \[
    h_s=\frac{h-sT_1}{1-s}\in\prof_1.
    \]
    In particular, we get 
    \[
    h=(1-s)h_s+sT_1,
    \] 
    that is, $h$ is not an extreme point of $\prof$. An analogous computation can be done when $h(0)>0$. 
    
    Assume now that $h(0)=h(1)=0$ and let $\nu$ be the positive Radon measure representing $-h''$. Since $h\ne T_{x_0}$ for every $x_0$, then there exists $y_0\in(0,1)$ such that $\nu([0,y_0])>0$ and $\nu((y_0,1])>0$. Let 
    \begin{align*}&\nu_1=\nu|_{[0,x_0]}, &\nu_2=\nu|_{(x_0,1]},\end{align*}
    and let $h_1,h_2$ be the solutions to
\begin{align*}&\begin{cases}
-h_1''=\nu_1, \\
h_1(0)=h_1(1)=0,
\end{cases} 
&\begin{cases}
-h_2''=\nu_2, \\
h_2(0)=h_2(1)=0.
\end{cases} 
\end{align*}
  We have that $h_1,h_2\in\prof$ and $h=h_1+h_2$, so that, letting
\[\tilde{h}_i=\dfrac{h_i}{\int_0^1 h_i\,dx},\qquad i=1,2,\]
we get $\tilde{h}_1,\tilde{h_2}\in\prof_1$, and
\[
h=t \tilde{h}_1+(1-t)\tilde{h}_2,
\]
with $t\in(0,1)$. Hence, $h$ is not an extreme point of $\prof_1$.
\end{proof}

Finally, we recall the definition of a quasiconcave function.
\begin{defi}
\label{def:quasiconvexs}
    A function $f:\R \to \R$ is quasiconcave if for all $x, y \in \R$ and $\lambda \in [0,1]$ we have

$$f(\lambda x + (1 - \lambda)y)\geq\min\big\{f(x),f(y)\big\}.$$

A function defined on an interval is quasiconcave if and only if it is monotone or 'increasing then decreasing', i.e. if there are two complementary intervals (one of which may be empty) such that it is increasing on the former and decreasing on the latter.
\end{defi}

\section{Minimization of the Steklov eigenvalue}
\label{sec: stek1}
For every $h\in\prof_1$ we consider the Sturm-Liouville eigenvalue $\sigma_1(h)$  defined in \eqref{eigensigma}

Lemmas \autoref{lem: compact} and \autoref{lem: continuity} prove that the problems
\[\begin{split}\max&\Set{\sigma_1(h)\colon h\in\prof_1}\\[15 pt]
\min&\Set{\sigma_1(h)\colon h\in\prof_1}
\end{split}\]
admit solutions. In particular, the solution to the maximization problem (see for instance \cite{T65}) is given by the parabola $p(x)=6x(1-x)$, with corresponding eigenvalue $\sigma_1(p)=12$. In this section, we aim to prove \autoref{teor: minstek}, namely that the problem
\[\min\Set{\sigma_1(h)\colon h\in\prof_1},\]
admits as unique solutions the functions $T_0(x)=2(1-x)$ and $T_1(x)=2x$ with corresponding eigenvalue 
\[
\sigma_1(T_0)=\sigma_1(T_1)=(j'_{0,1})^2/2,
\]
where $j'_{0,1}$ is the first positive zero of the first derivative of the Bessel function $J_0$.

\begin{oss}
The function 
\[h\in\prof\longmapsto \sigma_1(h),\]
satisfies the following properties:
\begin{itemize}
    \item \textbf{monotonicity}: for every $h_0,h_1\in\prof$, if $h_0\le h_1$ then 
    \[
    \sigma_1(h_0)\le\sigma_1(h_1);
    \]
    \item \textbf{homogeneity}: for every $h\in\prof$ and for every $\alpha>0$, 
    \[
    \sigma_1(\alpha h)=\alpha\sigma_1(h);
    \]
    \item \textbf{concavity}: for every $h_0,h_1\in\prof$ and for every $t\in[0,1]$, letting $h_t=(1-t)h_0+th_1$, we have that 
    \[
    \sigma_1(h_t)\ge (1-t)\sigma_1(h_0)+t\,\sigma_1(h_1);
    \]
    \item \textbf{symmetry}: let $h\in\prof$, and let $k(x)=h(1-x)$, then
    \begin{equation}
    \label{eq: symmetry}
    \sigma_1(k)=\sigma_1(h).
    \end{equation}
\end{itemize}
\end{oss}

\begin{prop}
\label{prop: minextr}
Let $h\in\prof_1$ be a solution to problem \eqref{eq:minstek}, then $h$ is an extreme point of $\prof_1$.
\end{prop}
\begin{proof}
    Let $h\in\prof_1$ be a solution to problem \eqref{eq:minstek}. By contradiction assume that $h$ is not an extreme point of $\prof_1$. Let $h_0,h_1\in\prof_1\setminus\set{h}$ and $t\in(0,1)$ such that 
    \[
    h=(1-t)h_0+th_1.
    \]
    Let $v\in H^1(0,1)$ be an eigenfunction for $\sigma_1(h)$ with 
    \[\int_0^1 v^2\,dx=1,\] then
    \[\begin{split}\sigma_1(h)&=\int_0^1 (v')^2 h\,dx=(1-t)\int_0^1 (v')^2 h_0\,dx+t\int_0^1 (v')^2 h_1\,dx\\[10 pt] &\ge(1-t)\sigma_1(h_0)+t\sigma_1(h_1).
    \end{split}\]
    On the other hand, by the minimality of $\sigma_1(h)$, we have 
\begin{align*}
    &\sigma_1(h_0)=\int_0^1 (v')^2 h_0\,dx, &\sigma_1(h_1)=\int_0^1 (v')^2 h_1\,dx.
\end{align*}
Therefore, $v$ is also an eigenfunction for $\sigma_1(h_0)$ and $\sigma_1(h_1)$. Let us now prove that $h_0=h$, thus reaching a contradiction. From the weak form of equation \eqref{eq:Steklov}, we have that for every $\varphi\in H^1(0,1)$ 
\[\begin{split}\int_0^1 v'\varphi' h\,dx&=\sigma_1(h)\int_0^1 v\varphi\,dx\\[10 pt]
     &=\sigma_1(h_0)\int_0^1 v\varphi\,dx=\int_0^1 v'\varphi' h_0\,dx,
\end{split}\]
that is
\[\int_0^1 (h-h_0)v'\varphi'\,dx=0\]
for every $\varphi\in H^1(0,1)$, which yields $h=h_0$, since, 
for every $\psi\in L^2(0,1)$, we can choose 
\[\varphi(x)=\int_0^x \psi(t)\,dt.\] 

Notice that we used that $v'$ has a constant sign. Indeed, it is well-known that the first eigenfunction of the Sturm-Liouville problem has exactly two nodal domains (see for instance \cite[Section VI.6]{courant_Hilbert}), hence, from equation \eqref{eq:Steklov} and the fact that the eigenfunction has zero mean, the claim follows.
\end{proof}
In order to study the minimum problem \eqref{eq:minstek}, we need to evaluate $\sigma_1$ on triangles, and we will need the following result, whose proof can be found in \cite{HM22}. 

\begin{lemma}
\label{lem: defsigma}
    Let $x_0\in [0,1]$. Then $\sigma_1(T_{x_0})$ is the first non-zero root $\sigma$ of the equation
    \begin{equation}\label{eq:sigmaT}J_0\left(\sqrt{2\sigma}x_0\right)J_0'\left(\sqrt{2\sigma}(1-x_0)\right)+J_0\left(\sqrt{2\sigma}(1-x_0)\right)J_0'\left(\sqrt{2\sigma}x_0\right)=0.\end{equation}
\end{lemma}
In addition, here we summarize the properties of the Bessel function $J_0$ which we will use.
\begin{prop}
\label{prop: bessel}
    Let $J_0$ be the Bessel function of the first kind of order 0, and let $j_{0,1}$ and $j_{0,1}'$ be the first zero of $J_0$ and $J_0'$ respectively. Then
    \[
        0<j_{0,1}<j_{0,1}',
    \]
    and
    \begin{alignat*}{2}
        &J_0(x) \ge 0 \qquad&&\forall \, x\in(0,j_{0,1}),\\[3 pt]
        &J_0'(x) \le 0 &&\forall \, x\in(0,j_{0,1}'),\\[3 pt]
        &J_0(x) < 0 &&\forall \, x\in(j_{0,1},j_{0,1}').
    \end{alignat*}
\end{prop}
\begin{proof}[Proof of \autoref{teor: minstek}]
By \autoref{lem: compact} and \autoref{lem: continuity} we have that the minimum problem \eqref{eq:minstek} admits a solution. On the other hand, by \autoref{prop: minextr} and \autoref{prop: triextr} we have that a solution to \eqref{eq:minstek} has to be a triangle $T_{x_0}$ for some $x_0\in[0,1]$. By the symmetry of $\sigma_1$ stated in \eqref{eq: symmetry}, we notice that to prove the theorem it is sufficient to show that the function
\[x_0\in\left[0,\frac{1}{2}\right]\mapsto \sigma_1(T_{x_0}),\] 
attains its minimum for $x_0=0$.

Let $j_{0,1}$ and $j'_{0,1}$ be the first positive roots of $J_0$ and $J'_0$ respectively. For every $x\in[0,1/2]$, and $s\in[0,+\infty)$, let
\[
F(x,s)=J_0(sx)J_0'(s(1-x))+J_0(s(1-x))J_0'(sx),
\]
which is the function defined in \autoref{lem: defsigma} that determines the value $\sigma_1(T_{x_0})$. Let $x_0\in(0,1/2)$ and let $s(x_0)$ be the smallest non-zero root of the equation
\begin{equation}\label{eq:ausiliaria}F(x_0,s)=0.\end{equation}
We claim that 
\begin{equation}
\label{eq: sx0inIx0}
s(x_0)\in I_{x_0}=\left(\dfrac{j_{0,1}}{(1-x_0)},\min\Set{\dfrac{j_{0,1}}{x_0},\dfrac{j'_{0,1}}{1-x_0}}\right).
\end{equation}
Indeed, since $J_0$ and $-J_0'$ are positive in $(0,j_{0,1})$,  and $x_0<1-x_0$, then
\[
F(x_0,s)<0 \qquad \forall \, s\in\left(0,\frac{j_{0,1}}{1-x_0}\right].
\]
On the other hand, using again the properties in \autoref{prop: bessel}, a direct computation gives
\[
F\left(x_0,\min\Set{\dfrac{j_{0,1}}{x_0},\dfrac{j'_{0,1}}{1-x_0}}\right)>0,
\] 
thus proving the claim. Notice that \eqref{eq: sx0inIx0} gives
\begin{equation}
\label{eq: sign}
\begin{aligned}
&J_0(s(x_0)x_0)>0, \qquad &J_0(s(x_0)(1-x_0))<0, \\[3 pt]
&J_0'(s(x_0)x_0)<0, \qquad &J_0'(s(x_0)(1-x_0))<0. 
\end{aligned}
\end{equation}
Since $J_0$ solves the equation
\begin{equation}
\label{eq: besselODE}
J_0''(t)+\frac{J_0'(t)}{t}+J_0(t)=0,
\end{equation}
then we have
\[\begin{split}
\partial_s F(x_0,s)=&J'_0(sx_0)J_0'(s(1-x_0))-J_0(sx_0)J_0(s(1-x_0))\\[8 pt]&-\dfrac{1}{s}\left(J_0(sx_0)J'_0(s(1-x_0))+J_0(s(1-x_0))J'_0(sx_0)\right).
\end{split}\]
In particular, \eqref{eq:ausiliaria} and \eqref{eq: sign} ensure that 
\begin{equation}
\label{eq: desposi}
\partial_s F(x_0,s(x_0))>0. 
\end{equation}
By the implicit function theorem, the function $x_0\mapsto s(x_0)$ is continuous, differentiable and
\[s'(x_0)\partial_s F(x_0,s(x_0))+\partial_{x} F(x_0,s(x_0))=0.\]
Using \eqref{eq: besselODE}, direct computations give
\[\partial_{x} F(x_0,s(x_0))=-\dfrac{J_0(s(x_0)(1-x_0))J'_0(s(x_0)x_0)}{x_0}+\dfrac{J(s(x_0)x_0)J'_0(s(x_0)(1-x_0))}{1-x_0}.\]
As before, \eqref{eq: sign} ensure that 
\begin{equation}
\label{eq: dexnega}
\partial_x F(x_0,s(x_0))<0.
\end{equation}
Joining \eqref{eq: desposi} and \eqref{eq: dexnega}, we have that $s'(x_0)>0$ and $x_0\mapsto s(x_0)$ is increasing. Finally, 
\[
\sigma_1(T_{x_0})=s^2(x_0)/2
\]
is increasing for $x_0\in(0,1/2)$, and the minimum is achieved when $x_0=0$. 
\end{proof}

\begin{oss}
    Equation \eqref{eq:sigmaT} for $x_0=0$ reduces to
\[J'_0\left(\sqrt{2\sigma}\right)=0\]
that is, $\sigma_1(2x)=\sigma_1(T_0)=(j'_{0,1})^2/2$.
\end{oss}

\section{An alternative proof for the minimum of \texorpdfstring{$\sigma(h)$}{s(h)}}
\label{sec: stek2}
In this section, we minimize $\sigma_1(h)$  using an alternative approach that avoids the explicit computation of the eigenvalue. In particular, our aim is to define a particular kind of symmetrization that allows us to prove that solutions to \eqref{eq:minstek} have to be monotone. Before defining the aforementioned symmetrization we prove an equivalent formulation for the eigenvalue $\sigma_1(h)$, referring to the ideas for the proof in \cite[Lemma 4.2]{OK90}

\begin{prop}
\label{prop: sigma=hardy}
    Let $h\in\prof_1$, then 
\[
    \sigma_1(h)=\min\Set{\dfrac{\displaystyle\int_0^1  (\varphi')^2 \,dx}{\displaystyle\int_0^1 \frac{\varphi^2}{h}\,dx}| \begin{aligned}&\:\,\varphi\in H^1_0(0,1), \\[3 pt] \displaystyle &\int_0^1 \frac{\varphi^2}{h}\,dx<\infty
\end{aligned}}.
\]
\end{prop}
\begin{proof}

    Let $v\in H^1(0,1)$ be a weak solution to \eqref{eq:Steklov} and let $w= h v'$. Then, since  $w'=-\sigma_1(h)v$, we have that $w\in H^2(0,1)\cap H^1_0(0,1)$ and that is a solution to
\begin{equation}
\label{cange_var eq}
\begin{cases}
    \displaystyle{-w''(x)= \frac{\sigma_1(h)}{h(x)}w(x) }\qquad\qquad x \in (0,1)\\[7 pt]
    w(0)=w(1)=0,
\end{cases}
\end{equation}
and 
\begin{equation}\label{eq:wsigma}\dfrac{\displaystyle\int_0^1  (w')^2 \,dx}{\displaystyle\int_0^1 \frac{w^2}{h}\,dx}=\sigma_1(h).\end{equation}
Following classical arguments (see for instance \cite{FNT12}) we have that $v$ vanishes in one and only one point $x_0\in(0,1)$, so that $w'$ vanishes only in $x_0$. Without loss of generality, we can assume that $w'$ is positive in $[0,x_0)$ and it is negative in $(x_0,1]$. Let now $\varphi\in H^1_0(0,1)$ be such that
\[\int_0^1 \dfrac{\varphi^2}{h}\,dx<+\infty.\]
Then, for every $0<x<x_0$, we have that
\[\begin{split}
\dfrac{1}{h(x)}\varphi^2(x)&\le \dfrac{1}{h(x)}\left(\int_0^x \dfrac{(\varphi'(t))^2}{w'(t)}\,dt\right)\left(\int_0^x w'(t)\,dt\right)\\[15 pt]
&=\dfrac{w(x)}{h(x)}\int_0^x \dfrac{(\varphi'(t))^2}{w'(t)}\,dt.
\end{split}\]
Then, integrating from $0$ to $x_1<x_0$, we get
\[\begin{split}
\int_0^{x_1}\dfrac{\varphi^2(x)}{h(x)}\,dx&\le\int_0^{x_1} \dfrac{w(x)}{h(x)}\int_0^x \dfrac{(\varphi'(t))^2}{w'(t)}\,dt\,dx\\[15 pt]
&=\int_0^{x_1}\dfrac{(\varphi'(t))^2}{w'(t)}\int_t^{x_1}\dfrac{w(x)}{h(x)}\,dx\,dt.
\end{split}\]
Using \eqref{cange_var eq}, and the fact that $w'(x_1)>0$, then we have
\[\int_0^{x_1}\dfrac{\varphi^2(x)}{h(x)}\,dx\le \dfrac{1}{\sigma_1(h)}\int_0^{x_1} (\varphi'(t))^2\,dt.\]
Letting $x_1$ go to $x_0$ we have
\begin{equation}
    \label{left}
    \int_0^{x_0}\dfrac{\varphi^2(x)}{h(x)}\,dx\le \dfrac{1}{\sigma_1(h)}\int_0^{x_0} (\varphi'(x))^2\,dx.
\end{equation}
Similar computations can be done in the case $x>x_0$, so that we have 
\begin{equation}
    \label{right}
     \int_{x_0}^1\dfrac{\varphi^2(x)}{h(x)}\,dx\le \dfrac{1}{\sigma_1(h)}\int_{x_0}^1 (\varphi'(x))^2\,dx.
\end{equation}
Joining \eqref{left} and \eqref{right} we have
\[\int_0^1\dfrac{\varphi^2(x)}{h(x)}\,dx\le \dfrac{1}{\sigma_1(h)}\int_0^1 (\varphi'(x))^2\,dx,\]
that is
\begin{equation}\label{eq:gesigma}\dfrac{\displaystyle\int_0^1  (\varphi')^2 \,dx}{\displaystyle\int_0^1 \frac{\varphi^2}{h}\,dx}\ge \sigma_1(h).\end{equation}
Since $w$ is an admissible function, the assertion follows from \eqref{eq:wsigma} and \eqref{eq:gesigma}.
\end{proof}

We now define the rearrangement mentioned above. Let
\[
w: [0,1]\to \R
\]
be a quasi-concave piecewise $C^1$ function such that
\[
\abs{\{w'=0\}}=0,  \qquad w(0)=w(1)=0,
\]
and let us denote by 
\[
w_M=\max_{[0,1]} w, 
\]
and by $x_M$ the maximum point of $w$. 
We aim to rearrange $w$ in such a way that the derivative of the rearranged function $w^\sharp$ concentrates at the left of the new maximum point $x_M^*$. 

For every $t \in (0,w_M)$, we define $(x_t,y_t):= \left\{ w(x) >t\right\}$, and the distribution functions
\begin{equation}
\label{eq: mu12}
\begin{split}
\eta_1(t)=x_M-x_t=\abs{\set{w>t}\cap (0,x_M)}, \\
\eta_2(t)=y_t-x_M=\abs{\set{w>t}\cap (x_M,1)}.
\end{split}
\end{equation}
Notice that 
\[
\eta_{1} : (0,w_M) \to (0,x_M), \qquad \qquad
\eta_{2} : (0,w_M) \to (0,1-x_M)
\]
are both decreasing, invertible, absolutely continuous functions, and that, for a.e. $t\in[0,1]$,
\begin{gather*}
	\eta_1' (t)= -\frac{1}{\abs{w'(x_t)}}, \qquad \qquad
		\eta_2' (t)= -\frac{1}{\abs{w'(y_t)}}.
\end{gather*}
Let us now define the rearranged distribution functions in such a way that, for a.e. $t\in[0,1]$,
\begin{equation}
	\label{min,max}
	\begin{gathered}
		\eta_{\ast,1}' (t)= \max\{	\eta_1' (t), 	\eta_2' (t) \},\\
		\eta_{\ast,2}' (t)=\min\{	\eta_1' (t), 	\eta_2' (t) \},
	\end{gathered}
\end{equation}
namely,  
\begin{equation}
\label{eq: mu12*}
\begin{split}
	\eta_{\ast,1} (t):= - \int_t^{w_M} 	 \max\{	\eta_1' (s), 	\eta_2' (s) \} \, ds,\\
	\eta_{\ast,2} (t):= - \int_t^{w_M} 	 \min\{	\eta_1' (s), 	\eta_2' (s) \} \, ds.
 \end{split}
\end{equation}

\begin{oss}
Here we emphasize some properties of these distribution functions:
\begin{itemize}
	\item for every $t\in(0,w_M)$, we have
         \[
           \eta_{1}(t)+\eta_{2}(t)= \eta_{\ast,1}(t)+\eta_{\ast,2}(t)=\abs{\{w>t\}};
         \]
    \item  by \eqref{min,max}, we have that, for a.e. $t\in(0,w_M)$,
    \[
    \begin{split}
        \frac{1}{\abs{\eta_{\ast,1}' (t)}}&= \max\left\{	\frac{1}{\abs{\eta_1' (t)}},\frac{1}{\abs{\eta_2' (t)}}\right\} \\[5 pt]
        &=\max\{\abs{w'(x_t)},\abs{w'(y_t)} \},
    \end{split}
    \]
    and
    \[
    \begin{split}
        \frac{1}{\abs{\eta_{\ast,2}' (t)}}&= \min\left\{	\frac{1}{\abs{\eta_1' (t)}},\frac{1}{\abs{\eta_2' (t)}}\right\} \\[5 pt]
        &=\min\{\abs{w'(x_t)},\abs{w'(y_t)} \}.
    \end{split}
    \]
   \item By \eqref{min,max}, we have 
    \begin{equation}
    \label{eq: 1/mi'+1/mi'}
      \frac{1}{\abs{\eta_{\ast,1}' (t)}^{\alpha}}+\frac{1}{\abs{\eta_{\ast,2}' (t)}^{\alpha}} =\frac{1}{\abs{\eta_{1}' (t)}^{\alpha}}+\frac{1}{\abs{\eta_{2}' (t)}^{\alpha}}
    \end{equation}
    for every $\alpha\in \R$.
    \item for $t=0$, we denote by
         \[
             x^\ast_M:=\eta_{\ast,1}(0)= 1-\eta_{\ast,2}(0),
         \]
    this will play the role of the maximum point for the rearranged function.
	\item the functions 
        \[
            \eta_{\ast,1} : (0,w_M) \to (0,x^\ast_M), \qquad \qquad	\eta_{\ast,2} : (0,w_M) \to (0,1-x^\ast_M)
        \]
        are decreasing, invertible, absolutely continuous functions.
\end{itemize}
\end{oss}
We now define the rearrangement $w^\sharp$ as follows:
\begin{defi}
	\label{rearr}
	Let $w$ be a quasi-concave piecewise $C^1$ function such that
\[
\abs{\{w'=0\}}=0,  \qquad w(0)=w(1)=0,
\]
    and let $\eta_{1},\eta_{2},\eta_{\ast,1}$ and $\eta_{\ast,2}$ be the functions defined in \eqref{eq: mu12}, and \eqref{eq: mu12*}. We define the competitor $w^\#$ as 
	\begin{equation*}
		w^\sharp (x)= 
		\begin{cases}
			w_m - \eta_{\ast,1}^{-1}(x) \quad &\text{ if } x \leq x^\ast_M, \\
			w_m - \eta_{\ast,2}^{-1}(1-x) \quad &\text{ if } x > x^\ast_M.
		\end{cases}
	\end{equation*}
\end{defi} 

\begin{figure}
    \includegraphics[width=\textwidth]{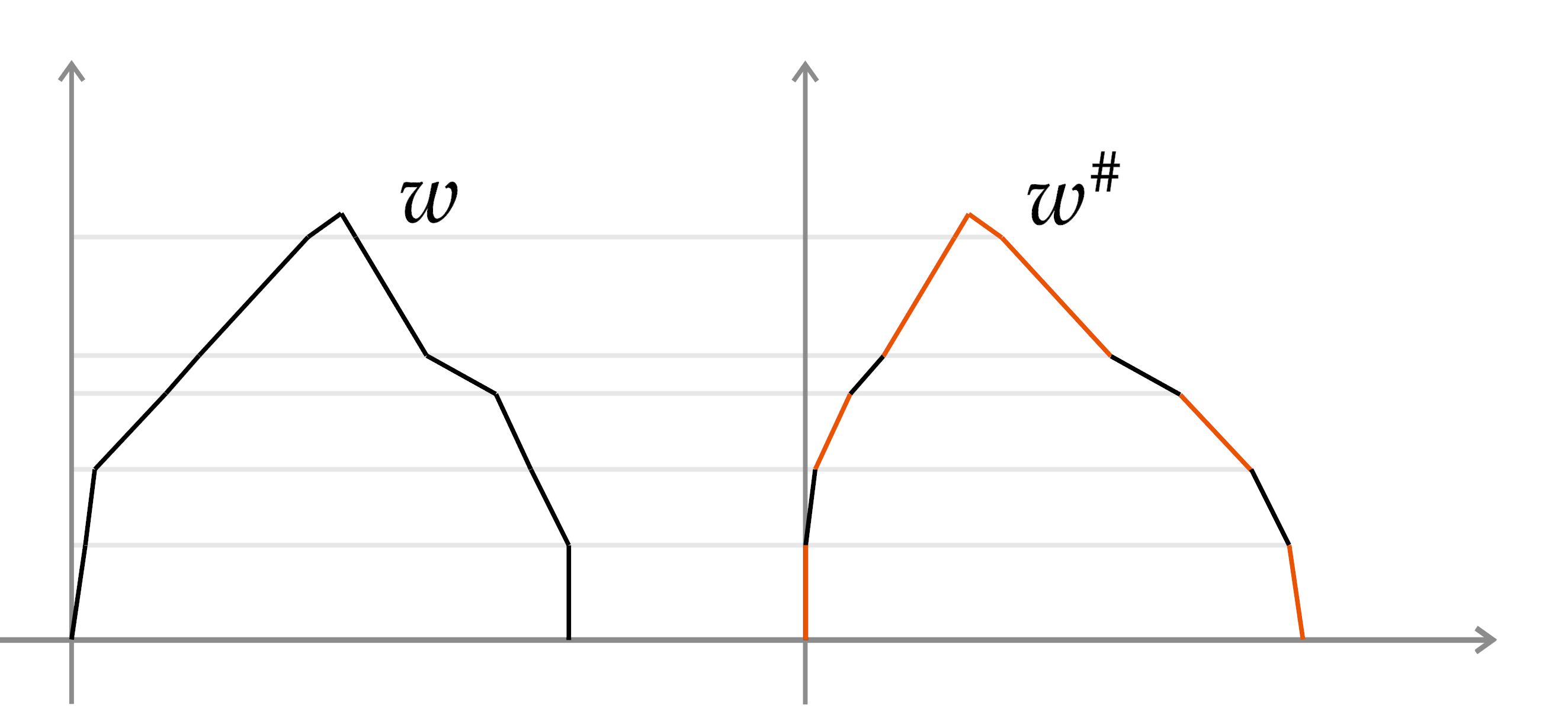}
    \caption{Function $w^\#$ when $w$ is a quasi-concave affine function}
\end{figure}

\begin{oss}
From the definition we have that $w^\sharp$ is increasing in $[0,x_M^\ast)$ and decreasing in $(x_M^\ast,0]$, so that $w^\sharp$ is quasi-concave. Moreover, we have that $w^\sharp$ and $w$ are equi-measurable, i.e.
\[
\norma{w^\sharp}_{L^p(0,1)}=\norma{w}_{L^p(0,1)}
\]
for every $p\in[1,+\infty]$. 
\end{oss}
We now prove some useful properties of this rearrangement.
\begin{lemma}
	\label{simmetric}
	Let $w$ be a quasi-concave piecewise $C^1$ function such that
    \[
    \abs{\{w'=0\}}=0,  \qquad w(0)=w(1)=0,
    \]
    and let $w^\sharp$ be the competitor defined in \autoref{rearr}. Then 
	$$
	w^\sharp(x)= (w(1-x))^\sharp.
	$$
\end{lemma}
\begin{proof}
	 Let us set $v(x)= w(1-x)$	and $\nu_1, \nu_2, \nu_{\ast,1}$ and $\nu_{\ast,2}$ the equivalent quantities defined for $v$. Then we have 
  \[
  \nu'_1(t)= \eta'_2(t), \qquad \qquad\nu'_2(t)= \eta'_1(t),
  \]
  and, in particular, 
  \[
  \nu'_{\ast,1}(t)= \eta'_{\ast,1}(t), \qquad \qquad\nu'_{\ast,2}(t)= \eta'_{\ast,2}(t).
  \]
\end{proof}

\begin{lemma}
	\label{rearrenggrad}
	Let $w$ be a quasi-concave piecewise $C^1$ function such that
   \[
\abs{\{w'=0\}}=0,  \qquad w(0)=w(1)=0,
\]
    and let $w^\sharp$ be its competitor defined in \autoref{rearr}. Then,
	\begin{equation}
    \label{eq: normep}
				\norma{(w^\sharp)'}_{L^p(0,1)}=\norma{w'}_{L^p(0,1)} \qquad  \forall \, p\ge 1.
	\end{equation}
\end{lemma}
\begin{proof}
	Let us compute separately the norms: by the coarea formula (see \autoref{coarea}), we get
	\begin{equation}
    \label{eq: normaw'}
        \begin{split}
    		 \int_0^1 \abs{w'(x)}^p\, dx &= \int_0^{w_M} \int_{\{w=t\}} \abs{w'(x)}^{p-1}\, d\mathcal{H}^0(x)\,dt\\[7 pt]
          &= \int_0^{w_M} \left(\abs{w'(x_t)}^{p-1}+ \abs{w'(y_t)}^{p-1}\right)\, dt\\[7 pt]
          &= \int_0^{w_M} \left(\frac{1}{\abs{\eta_1'(t)}^{p-1}}+ \frac{1}{\abs{\eta_2'(t)}^{p-1}}\right)\, dt.
      \end{split}
	\end{equation}
	Analogously,
	\begin{equation}
    \label{eq: normaww'}
		\int_0^1 \abs{(w^\sharp)'(x)}^p\, dx= \int_0^{w_M} \frac{1}{\abs{\eta_{\ast,1}'(t)}^{p-1}}+ \frac{1}{\abs{\eta_{\ast,2}'(t)}^{p-1}}\, dt.
	\end{equation}
Joining \eqref{eq: normaw'}, \eqref{eq: normaww'}, and \eqref{eq: 1/mi'+1/mi'}, we get \eqref{eq: normep}.
\end{proof}
We now state the property of $w^\sharp$ that will be crucial in the proof of \autoref{teor: minstek}.
\begin{lemma}
	\label{comaprison}
    Let $w$ be a quasi-concave piecewise $C^1$ function such that
\[
\abs{\{w'=0\}}=0,  \qquad w(0)=w(1)=0,
\]
    and let $w^\sharp$ be its competitor defined in \autoref{rearr}. Assume that
    \[
    h:(0,1)\to [0,+\infty)
    \]
    is a concave function, and let $h^\ast$ be its increasing rearrangement. Then
\begin{equation*}
	\int_{0}^1 \frac{w^2}{h}\,dx  \leq \int_{0}^1 \frac{(w^\sharp)^2}{h^\ast}\,dx.
\end{equation*} 
\end{lemma}
\begin{proof}
	By Fubini's theorem, we can write
	\begin{equation*}
		\int_0^1  \frac{w^2(x)}{h(x)}\, dx =   \int_0^1 w^2(x) \int_0^{\frac{1}{h(x)}} \, dt\,dx= \int_0^\infty \int_{\left\{\frac{1}{h(x)}>t\right\}}  w^2(x)\, dx\, dt.
	\end{equation*}
	The same computation leads to
    \[
        \int_0^1  \frac{(w^\sharp)^2(x)}{h^\ast(x)}\, dx =\int_0^\infty \int_{\left\{\frac{1}{h^\ast(x)}>t\right\}}  (w^\sharp)^2(x)\, dx\, dt.
    \]
    Hence, to prove the lemma it is sufficient to prove that for a.e. $t>0$
	\begin{equation*}
		 \int_{\left\{\frac{1}{h(x)}>t\right\}} w^2(x)\, dx 	\leq	 \int_{\left\{\frac{1}{h^\ast(x)}>t\right\}} (w^\sharp)^2(x)\, dx.
	\end{equation*}
For every $t\in(0,\norma{1/h}_\infty)$, let us define 
\[
D_t:=\left\{\frac{1}{h(x)}>t\right\}=(0,\tilde{x}_t) \cup (\tilde{y}_t,1),
\]
for some $\tilde{x}_t, \tilde{y}_t\in(0,1)$. In an analogous way, by the definition of increasing rearrangement (see \autoref{defi:incre}), we have
\[
D^\ast_t=\left\{\frac{1}{h^\ast(x)}>t\right\}= (0, \tilde{x}_t+1-\tilde{y}_t).
\]
Let $m=\min\{w(\tilde{x}_t), w(\tilde{y}_t)\}$, and let us define the following auxiliary functions
\[
	f(x)= \min\{w(x), m\}^2, \qquad \qquad	g(x)= (w^2-m^2)_+,
 \]
 so that 
\[
	\int_{D_t} w^2\,dx = \int_{D_t} f \,dx+ \int_{D_t} g\,dx.
\]
Similarly, we define 
\[
	f_0(x)= \min\{w^\sharp(x), m\}^2, \qquad \qquad
	g_0(x)= ((w^\sharp)^2-m^2)_+,
 \]
 so that
\[
	\int_{D_t^*} (w^\sharp)^2\,dx = \int_{D_t^*} f_0\,dx + \int_{D_t^*} g_0\, dx.
\]
 We now evaluate separately the two terms:
 \begin{enumerate}
		\item  By the definition of $m$, we have that
                \[
                w(x)>m \qquad \forall \, x\in (0,1)\setminus D_t.
                \]
                Therefore, since $f$ and $f_0$ are equi-measurable, we get
		\begin{equation}
        \label{eq: f<f0}
        \begin{split}
			\int_{D_t} f(x)\, dx &= \int_0^1 f(x)\,dx - (1-\abs{D_t})m^2 \\[7 pt]
                        &= \int_0^1 f_0(x)\, dx - \int_{(0,1) \setminus D^\ast_t} m^2\,dx \\[7 pt]
			&\leq \int_0^1 f_0(x)\,dx - \int_{(0,1) \setminus D^\ast_t} f_0(x)\,dx  \\[7 pt]
   &=\int_{D^\ast_t} f_0(x)\,dx,
        \end{split}
		\end{equation}
        where we have used that $\abs{D_t}=\abs{D_t^*}$, and that $m\ge f_0$;
		\item \autoref{simmetric} allows us to assume without loss of generality that $w(\tilde{y}_t)=m$. Therefore, the quasi-concavity of $w$ ensures that
        \[
        w(x)\le m \qquad\forall \, x\in(\tilde{y}_t,1),
        \]
        and we can write
		\begin{equation}
        \label{eq: intg}
		\int_{D_t} g(x)\, dx= \int_0^{\tilde{x}_t} (w^2(x)-m^2)_+\,dx = \int_{m}^{w_M} 2r \abs{\{w>r\} \cap (0,\tilde{x}_t)}\, dr.
		\end{equation}
		On the other hand,
		\begin{equation}
        \label{eq: intg0}
        \begin{split}
			\int_{D^\ast_t}g_0(x)\,dx &= \int_0^{\tilde{x}_t+1-\tilde{y}_t}g_0(x)\,dx \\[7 pt]
            &\geq \int_0^{\tilde{x}_t}g_0(x)\,dx \\[7 pt]
            &= \int_{m}^{w_M} 2r \abs{\{w^\sharp>r\} \cap (0,\tilde{x}_t)}\, dr.
        \end{split}
		\end{equation}

        We now claim that
		\begin{equation}
        \label{eq: misww}
		\abs{\{w^\sharp>r\} \cap (0,\tilde{x}_t)} \geq \abs{\{w>r\} \cap (0,\tilde{x}_t)}.
		\end{equation}
        Indeed, if we let 
        \[
         \{w>r\}=(x_r,y_r), \qquad \qquad \{w^\sharp>r\}=(x^*_r,y^*_r),
        \]
        then \eqref{min,max} gives
        \[
        x^\ast_r = -\int_0^r\eta'_{\ast,1} (s) \, ds   \leq -\int_0^r\eta'_{1} (s) \, ds=x_r,
        \]
    while the equi-misurability of $w$ and $w^\sharp$ gives
        \[
        y^\ast_r =(y^\ast_r-x^\ast_r)+x^\ast_r=(y_r-x_r)+x^\ast_r\le y_r.
        \]
        Therefore we get
        \[
        \begin{split}
            \abs{\{w^\sharp>r\}\cap (0,\tilde{x}_t)}=\abs{\{w>r\}\cap (0,\tilde{x}_t)} \qquad \qquad \text{if }y_r\le \tilde{x}_t, \\[5 pt]
            \abs{\{w^\sharp>r\}\cap (0,\tilde{x}_t)}>\abs{\{w>r\}\cap (0,\tilde{x}_t)} \qquad \qquad \text{if }y_r> \tilde{x}_t, 
         \end{split}
        \]
        thus the claim is proved. Finally, joining \eqref{eq: intg}, \eqref{eq: intg0}, and \eqref{eq: misww}, we have that
        \begin{equation}
        \label{eq: eq:intgg0}
			\int_{D^\ast_t}g_0(x)\,dx \geq 	\int_{D_t}g(x)\,dx,
		\end{equation}
 \end{enumerate}
 and the result follows from \eqref{eq: eq:intgg0}, and \eqref{eq: f<f0}.
\end{proof}

We now turn our attention to the eigenvalue problem. 

\begin{proof}[Alternative proof of \autoref{teor: minstek}]
Let $h\in\prof_1$, by \autoref{prop: sigma=hardy} we have that
\begin{equation}
\label{eq:varcar2}
\sigma_1(h)=\min\Set{\dfrac{\displaystyle\int_0^1  (\varphi')^2 \,dx}{\displaystyle\int_0^1 \frac{\varphi^2}{h}\,dx}\colon\, \varphi\in H^1_0(0,1)}.
\end{equation}
Let $w$ be a minimizer in \eqref{eq:varcar2}, then by \autoref{rearrenggrad}, and \autoref{comaprison}, we have
\begin{equation}\label{h>h_*}
	\sigma_1(h) = \ddfrac{\int_{0}^1 (w')^2}{\int_{0}^1 \frac{w^2}{h}} \geq \ddfrac{{\int_{0}^1 ((w^\sharp)')^2}}{{\int_{0}^1 \frac{(w^\sharp)^2}{h^\ast} }}\geq \sigma_1(h^\ast).
\end{equation}
By \autoref{prop: minextr}, and \autoref{prop: triextr}, we have that the minimum of $\sigma_1$ is a triangle $T_{x_0}$  for some $x_0\in[0,1]$. Let $h=T_{x_0}$, then $h_*=T_1$ and, from \eqref{h>h_*}, we have
\[\sigma_1(T_{x_0})\ge\sigma_1(T_1),\]
which concludes the proof.
\end{proof}

\section{Ratio \texorpdfstring{$\mu/\sigma$}{mu/sigma}}
\label{sec: mi/sigma}
In this section, we prove \autoref{prop: mu=sigma} and \autoref{teor: mu/sigma}. We begin by defining an operator $\mathcal{G}$ on $\prof$ as follows: let $h\in\prof$, and let 
\begin{equation}
\label{eq: H}
H(x)=\dfrac{1}{\int_0^1 h(t)\,dt}\int_0^x h(t)\,dt;
\end{equation}
we notice that $H$ is a strictly increasing function such that $H(0)=0$ and $H(1)=1$. We then define
\[
\mathcal{G}(h)(x)=h^2(H^{-1}(x)).
\]
\begin{lemma}
    Let $h\in\prof$. Then $\mathcal{G}(h)\in \prof$, and the map 
    \[
    \mathcal{G}: \prof \to \prof
    \]
    is invertible.
\end{lemma}
\begin{proof}
Since $h\in\prof$, then $h'$ is defined a.e. in $[0,1]$, and
$h'$ is decreasing. We also have that $H^{-1}$ is a locally Lipschitz function and
\begin{equation}
\label{Hmenounoderiv}
    \frac{d}{dx}H^{-1}(x)=\frac{1}{h(H^{-1}(x))}\int_0^1h(t)\,dt.
\end{equation}
Therefore, $\mathcal{G}(h)$ is a.e. differentiable and
\[
\frac{d}{dx}\mathcal{G}(h)(x)=2\alpha h'(H^{-1}(x)),
\]
where 
\[
\alpha=\int_0^1h(t)\,dt.
\]
Since $H^{-1}$ is an increasing function and $h'$ is decreasing, then $\mathcal{G}(h)$ is a concave function, and $\mathcal{G}(h)\in\prof$. 
\medskip 

Let $k\in\prof$ and define 
\begin{equation}
\label{eq: K}
K(x)=\dfrac{1}{\int_0^1 \frac{1}{\sqrt{k(t)}}\,dt}\int_0^x\frac{1}{\sqrt{k(t)}}\,dt,
\end{equation}
then we want to prove that
\begin{equation}
\label{eq: inverse}
\sqrt{k(K^{-1}(x))}=\mathcal{G}^{-1}(k)(x).
\end{equation}
First we prove that $\sqrt{k\circ K^{-1}}\in \prof$. By direct computation,
\[
\frac{d}{dx}\sqrt{k(K^{-1}(x))}=\frac{\beta k'(K^{-1}(x))}{2k(K^{-1}(x))} ,
\]
where
\[
\beta=\int_0^1\frac{1}{\sqrt{k(t)}}\,dt.
\]
This proves that $\sqrt{k\circ K^{-1}}$ is concave, since $K^{-1}$ is increasing and $h'/h$ is decreasing because of the concavity of $h$. On the other hand, to prove \eqref{eq: inverse}, we observe that with a change of variables we get
\[
\int_0^x\sqrt{k(K^{-1}(t))}\,dt=\frac{K(x)}{\int_0^1\frac{1}{\sqrt{k(t)}}\,dt},
\]
and by definition of $\mathcal{G}$ we get
\[
\mathcal{G}\left(\sqrt{k\circ K^{-1}}\right)(x)=k(x).
\]
\end{proof}
We now prove that $\mathcal{G}$ is the operator in \autoref{prop: mu=sigma}.
\begin{proof}[Proof of \autoref{prop: mu=sigma}]
    Let $v\in H^1(0,1)$ be a function such that
    \[
    \int_0^1 v(t)h(t)\,dt=0,
    \]
    and let $H$ denote the integral function defined in \eqref{eq: H}. The change of variables $H(t)=s$ yields
    \[
    \left(\int_0^1 h(t)\,dt\right)\int_0^1 v(t)\,h(t)\,dt=\int_0^1 v(H^{-1}(s))\,ds,
    \]
    \[
    \left(\int_0^1 h(t)\,dt\right)\int_0^1 (v')^2(t)\,h(t)\,dt=\int_0^1 (v')^2(H^{-1}(s))\,ds,
    \]
    and
    \[
    \left(\int_0^1 h(t)\,dt\right)\int_0^1 v^2(t)\,h(t)\,dt=\int_0^1 v^2(H^{-1}(s))\,ds.
    \]
    Let $w(x)=v(H^{-1}(x))$, then by \eqref{Hmenounoderiv},
    $$
    w'(x)=\left(\int_0^1h(t)\, dt \right)\, v'(H^{-1}(x))\,\bigl(\mathcal{G}(h)(x)\bigr)^{-\frac{1}{2}}.
    $$
    Hence,  
     \[
    \ddfrac{\int_0^1 (v')^2(t)\,h(t)\,dt}{\int_0^1 v^2(t)\,h(t)\,dt}=\left(\int_0^1 h(t)\,dt\right)^{-2}\ddfrac{\int_0^1 (w')^2(t)\,\mathcal{G}(h)(t)\,dt}{\int_0^1 w^2(t)\,dt}.
    \]
Choosing $v=v_\mu$ to be the eigenfunction of $\mu_1(h)$, then we get
\[
\mu_1(h)\ge \left(\int_0^1 h(t)\,dt\right)^{-2}\sigma_1(\mathcal{G}(h)).
\]
On the other hand, choosing $w=w_\sigma$ to be the eigenfunction of $\sigma_1(\mathcal{G}(h))$, we get 
\[
\mu_1(h)\le\left(\int_0^1 h(t)\,dt\right)^{-2}\sigma_1(\mathcal{G}(h)),
\]
which gives \eqref{eq: neumtosigma}.

Let $k\in\prof$ and let $K$ be the integral function defined in \eqref{eq: K}. If we evaluate the integral on the right-hand side by means of the change of variables $t=K(s)$, we finally get
\[
\int_0^1 h(t)\,dt=\left(\int_0^1 \frac{1}{\sqrt{k(t)}}\,dt\right)^{-1},
\]
which gives \eqref{eq: sigmatoneum}.
\end{proof}
The following punctual estimate will be crucial.
\begin{prop}
    \label{prop: g<2h}
    Let $h\in\prof$. Then 
    \[
        \left(\int_0^1 h(t)\,dt\right)^{-1}\mathcal{G}(h)(x)\le 2h(x).
    \]
\end{prop}  
\begin{proof}
    Up to rescaling $h$, we can assume without loss of generality that $h\in\prof_1$. Notice, in addition, that if $h\equiv 1$, then the proof is trivial. Therefore, let $h\in \prof_1$ and $h\not\equiv 1$, and define 
    \[
    H(x)=\int_0^x h(t)\, dt.
    \]
    We claim that there exists a unique $\bar{x}\in[0,1]$ such that
    \begin{equation}
    \label{eq: Hbarx}
    \begin{aligned}
    H(x)\le x \qquad \qquad& x\in [0,\bar{x}], \\
    H(x)\ge x \qquad \qquad& x\in [\bar{x},1].
    \end{aligned}
    \end{equation}
    Indeed, if we denote by
    \[
    f(x)=H(x)-x,
    \]
    then, by the concavity of $h$ and the integral constraint, we have that the equation $h=1$ admits at most two solutions ($h$ cannot be equal to 1 in an entire interval, otherwise the concavity would give $\norma{h}_1<1$). Therefore, we have that there exist two points $x_1\in[0,1)$, and $x_2\in (0,1]$ such that
     \begin{align*}
    f'(x)< 0 \qquad \qquad &x\in [0,x_1)\cup(x_2,1], \\
    f'(x)>0  \qquad \qquad &x\in (x_1,x_2).
    \end{align*}
    Finally, noticing that $f(0)=f(1)=0$, then we have that there exists a unique zero $\bar{x}$ of $f$ in the interval $[x_1,x_2]$, thus the claim is proved. In particular, we have that
    \begin{align*}
    H^{-1}(x)\ge x \qquad \qquad& x\in [0,\bar{x}], \\
    H^{-1}(x)\le x \qquad \qquad& x\in [\bar{x},1],
    \end{align*}
    and $h(\bar{x})>1$.
    
    \medskip 
    These estimates allow us to compare the derivatives of $\mathcal{G}(h)$ and $h$. Denoting by $g(x)=\mathcal{G}(h)(x)$, we have that
    \begin{align}
    \label{eq: stimag'1}
        g'(x)=2h'(H^{-1}(x))\le 2h'(x) \qquad \qquad& x\in [0,\bar{x}], \\
    \label{eq: stimag'2}
        g'(x)=2h'(H^{-1}(x))\ge 2h'(x)  \qquad \qquad&x\in [\bar{x},1].
    \end{align}
    We recall that, as in \eqref{eq: h<2}, the concavity of $h$ ensures that
    \begin{equation*}
        {\norma{h}_\infty}\le 2.
    \end{equation*}
    Therefore, we get
    \[
    g(0)=h^2(0)\le 2 h(0),
    \]
    and by \eqref{eq: stimag'1},
    \[
        g(x)\le 2h(x) \qquad \qquad x\in[0,\bar{x}].
    \]
    Analogously,
    \[
        g(1)=h^2(1)\le 2h(1),
    \]
    and, by \eqref{eq: stimag'2},
    \[
        g(x)\le 2h(x) \qquad \qquad x\in[\bar{x},1].
    \]
\end{proof}

\begin{prop}\label{g=2h}
Let $h\in\prof_1$. 
\[\mathcal{G}(h)=2 h\]
if and only if $h=T_{x_0}$ for some $x_0\in[0,1]$.    
\end{prop}
\begin{proof}
By direct computation, one can prove that if $h=T_{x_0}$ for some $x_0\in[0,1]$, then 
\[h^2(x)=2h(H(x)).\]
Let us now assume that $\mathcal{G}(h)=2h$. Notice that, if $y\in[0,1]$ is a fixed point of the integral function $H$, then
\begin{equation}
\label{eq: h^2}
   h^2(y)=h^2(H^{-1}(y))=\mathcal{G}(h)(y)=2h(y),
\end{equation}
so that either $h(y)=0$ or $h(y)=2$. In particular, if $h(y)=2$, then by the concavity of $h$, we have that 
\[
h(x)=T_y(x) \qquad \forall \, x\in[0,1].
\]
Since $0$ and $1$ are always fixed points of $H$, if either $h(0)=2$ or $h(1)=2$ the assertion is proved. Therefore, let us assume that
\[
h(0)=0=h(1),
\]
then the equation $h=1$ admits at least two distinct solutions $0<x_1<x_2<1$ and, arguing as in the proof of \autoref{prop: g<2h}, we have that there exists  a fixed point $\bar{x}\in[x_1,x_2]$ for the function $H$ and, by \eqref{eq: h^2}, necessarily  $h(\bar{x})=2$ and $h=T_{\bar{x}}$.
\end{proof}

\begin{prop}
\label{prop: mu/sigma}
Let $h\in \prof_1$, then 
\begin{equation}\label{mu/sigma<2}
\frac{\mu_1(h)}{\sigma_1(h)}\le 2
\end{equation}
and the equality holds if and only if $h = T_{x_0}$ for some $ x_0\in [0, 1]$.
\end{prop}

\begin{proof}
 Let $w$ be an eigenfunction for $\sigma_1(h)$. Using Theorem \autoref{prop: mu=sigma},  \autoref{prop: g<2h},  and the variational characterization of $\sigma_1(\mathcal{G}(h))$, we obtain  
\begin{equation}
    \label{mule2sigma}
    \mu_1(h)=\sigma_1(\mathcal{G}(h))\leq\ddfrac{\int_0^1 (w')^2 \mathcal{G}(h)\,dx}{\int_0^1 w^2\,dx}\le\ddfrac{2\int_0^1 (w')^2 h\,dx}{\int_0^1 w^2\,dx}=2\sigma_1(h),
\end{equation}
thus proving \eqref{mu/sigma<2}.
Assume now that for some $h\in\prof_1$ equality holds, then by \eqref{mule2sigma} we have 
\begin{equation}\label{eq =}
\int_0^1 (w')^2 (\mathcal{G}(h)-2h)\,dx=0.
\end{equation}
Since $\mathcal{G}(h)\le2h$, then \eqref{eq =} yields $\mathcal{G}(h)=2h$, and \autoref{g=2h} ensures that $h=T_{x_0}$ for some $x_0\in[0,1]$.
\end{proof}
\begin{oss}
    Since it is not possible in general to have that $\mathcal{G}(h)\ge h$, then the same argument cannot be used for the lower bound
    \[
        \frac{\mu_1(h)}{\sigma_1(h)}\ge 1.
    \]
    For instance, let
    \[
        h(x)=\frac{1}{2}+x,
    \]
    then $\mathcal{G}(h)(0)=h^2(0)<h(0)$, while $\mathcal{G}(h)(1)=h^2(1)>h(1)$.
\end{oss}
Here we prove the lower bound in \autoref{teor: mu/sigma} in the symmetric case.
\begin{prop}
\label{prop: lower}
 Let $h \in \prof_1$ such that $h(1-x)=h(x)$ for all $x \in [0,1]$. Then
    \begin{equation}
    \label{eq: mu/sigma>1}
    \frac{\mu_1(h)}{\sigma_1(h)} \geq 1.  
    \end{equation}
\end{prop}
\begin{proof}

    Let $g=\mathcal{G}(h)$. By the variational characterization \eqref{eq:varcar2} of $\sigma_1$, and \autoref{prop: mu=sigma}, we can find a function $w\in H^2(0,1)$, symmetric with respect to $x=1/2$, such that
    \begin{equation}
        \label{eq: mu=sigma}
        \mu_1(h)=\sigma_1(g)=\ddfrac{\int_0^1 (w')^2(x)\,dx}{\int_0^1 \frac{w^2(x)}{g(x)}\,dx},
    \end{equation}
    and $w$ solves the problem
    \begin{equation*}
        \begin{cases}
            \displaystyle{-w''(x)= \frac{\sigma_1(h)}{h(x)}w(x) }\qquad\qquad x \in (0,1),\\[7 pt]
            w(0)=w(1)=0.
        \end{cases}
        \end{equation*}
    We can choose $w$ to be positive and concave, so that
    \begin{equation}
    \label{eq: wincreasing}
    \begin{aligned}
        w'(x)\ge 0 \qquad \qquad &\text{in }\left(0,\frac{1}{2}\right),\\
        w'(x)\le 0 \qquad \qquad &\text{in }\left(\frac{1}{2},1\right).
    \end{aligned}
    \end{equation}
    Moreover, by the variational characterization \eqref{eq:varcar2}, we get
    \begin{equation}
        \label{eq: sigma<mu}
        \sigma_1(h)\le \ddfrac{\int_0^1 (w')^2(x)\,dx}{\int_0^1 \frac{w^2(x)}{h(x)}\,dx},
    \end{equation}
    and then, joining \eqref{eq: mu=sigma} and \eqref{eq: sigma<mu}, we get
    \[
        \frac{\mu_1(h)}{\sigma_1(h)}\ge \ddfrac{\int_0^1 \frac{w^2(x)}{h(x)}\,dx}{\int_0^1 \frac{w^2(x)}{g(x)}\,dx}.
    \]
    To prove \eqref{eq: mu/sigma>1} it is sufficient to prove that
    \[
        \int_0^1 \frac{w^2(x)}{g(x)}\,dx\le \int_0^1 \frac{w^2(x)}{h(x)}\,dx.
    \]
    We now compute the left-hand side by means of the change of variables $x=H(y)$, where 
    \[
    H(y)=\int_0^y h(t)\,dt,
    \]
    so that
    \[
        \int_0^1 \frac{w^2(x)}{g(x)}\,dx=\int_0^1 \frac{w^2(H(y))}{h(y)}\,dy.
    \]
    We now notice that the symmetry of $h$ gives \eqref{eq: Hbarx} with $\bar{x}=1/2$, namely
    \begin{equation}
    \label{eq: H1/2}
    \begin{aligned}
    H(y)\le y \qquad \qquad& x\in \left[0,\frac{1}{2}\right], \\
    H(y)\ge y \qquad \qquad& x\in \left[\frac{1}{2},1\right].
    \end{aligned}
    \end{equation}
    Finally, joining \eqref{eq: H1/2}, and \eqref{eq: wincreasing}, we have
    \[
       \int_0^1 \frac{w^2(x)}{g(x)}\,dx=\int_0^1 \frac{w^2(H(y))}{h(y)}\,dy\le\int_0^1 \frac{w^2(y)}{h(y)}\,dy,
    \]
    which concludes the proof.
\end{proof}
\begin{proof}[Proof of \autoref{teor: mu/sigma}]
The result follows from \autoref{prop: mu/sigma} and \autoref{prop: lower}
\end{proof}
\subsection*{Acknowledgements}
We would like to thank Carlo Nitsch and Cristina Trombetti for the valuable advice that helped us to achieve these results.

The three authors were partially supported by Gruppo Nazionale per l’Analisi Matematica, la Probabilità e le loro Applicazioni
(GNAMPA) of Istituto Nazionale di Alta Matematica (INdAM).

\vspace{5mm}

\subsection*{Competing Interests}
The authors report there are no competing interests to declare. 

\vspace{1cm}

\addcontentsline{toc}{chapter}{Bibliografia}

\printbibliography[heading=bibintoc, title={References}]

\Addresses

\end{document}